\documentclass[preprint,12pt]{elsarticle}

\usepackage{amscd}
\usepackage{amsmath}
\usepackage{amsthm}
\usepackage{amsfonts}
\usepackage{graphicx}
\usepackage{amssymb}
\usepackage{color}
\usepackage{amsbsy}
\usepackage{graphicx}
\usepackage{amssymb,color,amsbsy}

\usepackage{float}
\usepackage{comment}

\usepackage[ruled]{algorithm2e}

\usepackage[matrix,arrow]{xy}
\usepackage{mathabx}   

\DeclareMathAlphabet{\mathpzc}{OT1}{pzc}{m}{it}

\usepackage{pgfpages}
\usepackage{tikz}
\usetikzlibrary{shapes.geometric}
\usetikzlibrary{decorations.pathmorphing}
\usetikzlibrary{arrows}
\usetikzlibrary{backgrounds}
\usetikzlibrary{positioning}
\usetikzlibrary{fit}
\usepackage{caption}
\usepackage{amsthm}

\usepackage{amsmath}

\usepackage[pagebackref=true, bookmarksopen=true,colorlinks=true, linkcolor=red,citecolor=blue]{hyperref}

\usepackage[capitalise]{cleveref}  

\usepackage{tikz-cd}

\usepackage{tkz-graph}

\usepackage{multicol}

\usepackage{subcaption}

\global\long\def\ñ{\sim}%

\newtheorem{theorem}{Theorem}[section]

\newtheorem{corollary}[theorem]{Corollary}

\newtheorem{lemma}[theorem]{Lemma}

\newcommand{\perm}[1]{\operatorname{perm}\left({#1}\right)}

\newcommand{\Sachs}[1]{\operatorname{Sachs}(#1)}

\newcommand{\fp}[1]{\operatorname{SD}(#1)}
\newcommand{\ke}[1]{\operatorname{KE}(#1)}

\newcommand{\Rea}[2]{R(#1,#2)}

\begin{document}

\begin{abstract}
	In this work it is shown that the SD-KE decomposition is multiplicative under determinantal-type functions for graphs with perfect matchings, providing a new tool for the study of unimodular and singular matchable graphs.	
\end{abstract}


\begin{frontmatter}
	
	\title{Generalized Edmonds-Sterboul-Deming configurations\\ Part 3:  Determinantal multiplicativity of the SD-KE decomposition of matchable graphs}
	
	\author[pan,daj]{Daniel A. Jaume}
	\ead{djaume@unsl.edu.ar}
	
	\author[pan,daj]{Diego G. Martinez}
	\ead{dgmartinez@unsl.edu.ar}
	\author[pan]{Cristian Panelo}
	\ead{crpanelo@unsl.edu.ar}
	
	\author[pan,daj]{Kevin Pereyra}
	\ead{kdpereyra@unsl.edu.ar}
	
	\address[pan]{Universidad Nacional de San Luis, Argentina.}
	\address[daj]{IMASL-CONICET, Argentina.}
	
	\date{Received: date / Accepted: date}

\end{frontmatter}
%
%
%

\section{Introduction}

	In 1985, Godsil \cite{godsil1985inverses} proved that the determinant of trees with perfect matchings is always $\pm 1$, depending on the parity of the matching number. Using the same technique, Yang and Ye \cite{yang2018inverses} extended Godsil's result in 2018 to bipartite graphs with a unique perfect matching. More recently, in 2025, Jaume, Martinez and Panelo~\cite{jaume2025determinantal} introduced a new tool, the notion of $M\!S$-alternating paths, which allowed them to generalize these earlier results to the setting of arbitrary graphs with a unique perfect matching.  
	
	In \cite{jaume2025determinantal}, it was also introduced a decomposition of graphs with a unique perfect matching into barbell subgraphs and K\H{o}nig--Egerv\'{a}ry graphs. This decomposition was later extended to general graphs by Jaume and Molina~\cite{jaumeSDKE2025} through a generalization of the classical flower and posy configurations of Edmonds \cite{edmonds1965paths}, Sterboul \cite{sterboul1979characterization}, and Deming \cite{deming1979independence}, as developed by Jaume, Panelo and Pereyra in \cite{kevin2025posy}.
	
	The SD--KE decomposition splits any graph into two disjoint subgraphs: a K\H{o}nig--Egerv\'{a}ry (KE) part and a Sterboul--Deming (SD) part, where every vertex lies in either a flower or a posy configuration. This decomposition is additive with respect to both the matching number and the independence number.  
	
	Our main contribution is to show that the SD--KE decomposition, when applied to graphs with a perfect matching, induces a multiplicative factorization of the determinant of the adjacency matrix. Moreover, this factorization provides a new approach to two outstanding problems: the characterization of unimodular graphs \cite{akbari2007unimodular} and the structural characterization of singular graphs \cite{von1957spektren}.  
	
	A central notion in this work is that of $M\!S$-alternating paths. We analyze how these paths are constrained by the presence of generalized configurations. Our results show that the flexibility of Jflowers and Jposies---in contrast to their classical counterparts---enables more natural composition and factorization arguments, which are crucial in the determinantal setting.  
	
	This paper is organized as follows. In \cref{Sec_Preliminaries}, we introduce notation and recall key concepts from the SD--KE decomposition. In \cref{Sec_Matching}, we develop the structural study of perfect matchings. Using a reachability lemma, we gain insight into the interplay between matchings and SD-graphs. Furthermore, we provide a polynomial-time algorithm for computing the SD--KE decomposition of graphs with a perfect matching. Finally, in \cref{Main}, we build on these developments to present our main result.

\section{Notation and preliminaries}\label{Sec_Preliminaries}

	For all graph-theoretic notions not defined here, the reader is referred to \cite{diestel2000graph}. All graphs in this work are simple: finite, undirected, without loops, and without multiple edges.
	
	Let $G$ be a graph. The vertex set and the edge set of $G$ are denoted by $V(G)$ and $E(G)$, respectively. The order of $G$ is $|G| := |V(G)|$. If $\{u,v\} \in E(G)$, we write $uv$ instead of $\{u,v\}$.
	
	If $H$ is a subgraph of $G$, we write $H \leq G$. If $X \subset V(G)$, then $G[X]$ denotes the subgraph of $G$ induced by $X$. If $F \leq G$, then $G - F$ is the subgraph of $G$ with vertex set $V(G) - V(F)$ and edge set $E(G) - E(F)$. We also write $G - e$, $G - v$, and $G - H$ for the subgraphs of $G$ obtained by deleting an edge $e \in E(G)$, a vertex $v \in V(G)$, or a subgraph $H \subseteq G$, respectively. 
	
	A \textit{matching} $M$ in $G$ is a set of pairwise nonadjacent edges. A vertex $v \in V(G)$ is \textit{saturated} by $M$ if $v$ is an endpoint of some $e \in M$; otherwise, $v$ is \textit{unsaturated}. A matching defines an involution in $G$: $M(v) = u$ if $uv \in M$, and $M(v) = v$ if $v$ is unsaturated by $M$. The \textit{matching number} of $G$, denoted $\mu(G)$, is the maximum cardinality of a matching of $G$.
	
	A matching $M$ is \textit{perfect} if every vertex is matched, i.e., $M(v) \ne v$ for all $v \in V(G)$. A graph with perfect matching is called \emph{matchable}, see \cite{lucchesi2024perfect}.

	A \textit{walk} in $G$ is a sequence of vertices $v_1 v_2 \cdots v_k$ such that $v_iv_{i+1} \in E(G)$ for all $i \in \{1, \dots, k-1\}$. Repetition of vertices and edges is allowed. A \textit{path} in $G$ is a walk in which all vertices are distinct. Given two paths $P = u \cdots v$ and $Q = w \dots x$, their concatenation  (if $vw \in E(G)$) is written $PQ := u \cdots v w \cdots x$. The ``inverse'' path from \(v\) to \(u\) is denoted by $P^{-1}$. If $G$ is connected and $u, v \in V(G)$, we write $uPv$ to denote a path between $u$ and $v$. 
	
	A \textit{cycle} in $G$ is called \textit{even} if it has an even number of edges. An \textit{even alternating cycle} with respect to a matching $M$ is an even cycle whose edges alternate between belonging and not belonging to $M$.

	Let $M$ be a matching in $G$. A path (or walk) is called \textit{alternating} with respect to $M$ if, for each pair of consecutive edges in the path, exactly one of them belongs to $M$. If the matching is clear from context, we simply say that the path is alternating. Given an alternating path (or walk) $P$ in a graph with a unique perfect matching, we say that $P$ is: \textit{\(M\)-\(mm\)-alternating} if it starts and ends with edges in $M$, \textit{\(M\)-\(nn\)-alternating} if it starts and ends with edges not in $M$, \textit{\(M\)-\(mn\)-alternating} if it starts with an edge in $M$ and ends with one not in $M$, and \textit{\(M\)-\(nm\)-alternating} if it starts with an unmatched edge and ends with a matched edge.
	
	A graph $G$ is a \emph{K\H{o}nig-Egerv\'ary graph} if $\alpha(G) + \mu(G) = |G|$, or equivalently, if $\tau(G) = \mu(G)$, where $\alpha(G)$ is the independence number of \(G\) and $\tau(G)$ is the covering number of \(G\). 
	
	Edmonds \cite{edmonds1965paths} defines the following configurations for a graph \( G \) relative to a matching \( M \). An \( M \)-\emph{blossom} is an odd cycle of length \( 2k + 1 \) in \( G \) that contains exactly \( k \) edges from \( M \). The \emph{base} of a blossom is the unique vertex that is not matched by \( M \) to another vertex within the blossom. Sterboul in \cite{sterboul1979characterization} and Deming in \cite{deming1979independence} introduced a configurations called  \textit{posy}. 	Given a graph \(G\) and a maximum matching \(M\) of \(G\) an \(M\)-posy is the configuration consisting of two different (not necessarily disjoint) \(M\)-blossoms joined by an odd length $mm$-alternating path, with respect $M$, such that the endpoints of the $mm$-alternating path are the bases of the two \(M\)-blossoms. The notion of Jposies was introduced in \cite{kevin2025posy}. An $M$-Jposy is a configuration consisting of two (not necessarily distinct) \( M \)-blossoms joined by an odd-length \( M \)-alternating walk such that the endpoints of the walk are the bases of the two \( M \)-blossoms. In \cref{fig:Jposy}, a maximum matching \( M \) of the graph \( G \) is shown in red. With respect to this matching, there are several \( M \)-Jposies in \( G \). The one highlighted in cyan in \cref{fig:Jposy} is formed by the blossom \( 9,5,4,9 \), whose base is vertex \( 9 \) (appearing twice), and the \( M \)-alternating walk\footnote{When we label the vertices of a graph with integers, we write paths and walks using commas.}:
	\[
	9,8,5,4,3,2,0,1,2,3,4,5,1,0,6,7,8,9.
	\]

	\begin{figure}[H]
		\centering
		\def \Vescala{1}
		\def \opacidad{0.2}
		\begin{tikzpicture}[commutative diagrams/every diagram,thick,scale=0.6] 
			
			\node[draw,black,fill,circle,scale=0.5,label=below:\small{$0$}] (0) at (-2,1) {};
			\node[draw,black,fill,circle,scale=0.5,label=above:\small{$1$}] (1) at (-4,2) {};
			\node[draw,black,fill,circle,scale=0.5,label=left:\small{$2$}] (2) at (-5,1) {};
			\node[draw,black,fill,circle,scale=0.5,label=left:\small{$3$}] (3) at (-5,-1) {};
			\node[draw,black,fill,circle,scale=0.5,label=below:\small{$4$}] (4) at (-4,-2) {};
			\node[draw,black,fill,circle,scale=0.5,label=left:\small{$5$}] (5) at (-3,0) {};
			\node[draw,black,fill,circle,scale=0.5,label=above:\small{$6$}] (6) at (-1,3) {};
			\node[draw,black,fill,circle,scale=0.5,label=right:\small{$7$}] (7) at (0,2) {};
			\node[draw,black,fill,circle,scale=0.5,label=below:\small{$8$}] (8) at (0,0) {};
			\node[draw,black,fill,circle,scale=0.5,label=below:\small{$9$}] (9) at (-2,-1) {};
			\node[draw,black,fill,circle,scale=0.5,label=below:\small{$10$}] (10) at (2,0) {};
			\node[draw,black,fill,circle,scale=0.5,label=right:\small{$11$}] (11) at (2,2) {};
			
			\node[opacity=0] (phantom) at (0,-3) {};
			
			\foreach \from/\to in {1/2,1/5,3/4,0/6,7/8,9/5,9/4,0/5,5/8,2/0,10/8} {
				\path[draw] (\from) -- (\to);
			}
			\foreach \from/\to in {2/3,4/5,1/0,6/7,8/9,{10/11}} {
				\path [draw, decorate, decoration={snake,amplitude=4,segment length=1.5mm, amplitude=2}, very thick, red] (\from) -- (\to);
				}
			\draw[orange, opacity=2*\opacidad, line width=10pt, rounded corners] (9) -- (5) -- (4) -- (9) -- (8) -- (5)-- (4)-- (3)-- (2)-- (0)-- (1)-- (2)-- (3)-- (4)-- (5)-- (1)-- (0)-- (6)-- (7)-- (8)-- (9);
		\end{tikzpicture}
		\caption{\(M\)-Jposy} \label{fig:Jposy}
	\end{figure}
	
	\begin{lemma}[\cite{kevin2025posy}]\label[lemma]{lem:noKE1}
		If \( G \) contains a $M$-Jposy, for \( M \in \mathcal{M}(G) \), then \( G \) is not a K\H{o}nig–Egerv\'{a}ry graph.
	\end{lemma}
	
	Let $G$ be a matchable graph. Define $V_{SD}(G)$ as the set of vertices of $G$ that belong to an $M$-Jposy for some $M \in \mathcal{M}(G)$. Following the decomposition introduced in \cite{jaumeSDKE2025}, the \emph{SD-KE decomposition} of $G$ consists of two subgraphs, denoted by $\fp{G}$ and $\ke{G}$, where
	\[
	\fp{G} = G[V_{SD}(G)] \quad \text{and} \quad \ke{G} = G - \fp{G}.
	\]
	We refer to $\fp{G}$ as the \emph{SD-part} of $G$ and to $\ke{G}$ as the \emph{KE-part} of $G$. A graph $G$ is called a \emph{Sterboul--Deming graph}, or \emph{SD-graph} for short, if $\ke{G} = \emptyset$. The sets  \(V_{SD}(G)\) and \(V_{KE}(G):=V(\ke{G})\) are called the \emph{SD-KE partition} of \(G\) (even though one of them could be empty).

	A spanning subgraph $S$ of $G$ is called a \emph{Sachs subgraph} of $G$ if every component of $S$ is either $K_2$ or a cycle. Let $\Sachs{G}$ denote the set of all Sachs subgraphs of $G$. If $G$ has a perfect matching, then for every $M \in \mathcal{M}(G)$ we have $M \in \Sachs{G}$. In particular, if $G$ has a perfect matching, then $|\Sachs{G}| \geq 1$. The following is a classical result.

	\begin{theorem}[\cite{harary1962determinant}]
		Let $G$ be a with perfect matching graph. Then:
		\begin{align*}
			\det(G)  &= \sum_{S \, \in \, \Sachs{G}}(-1)^{\kappa_e(G)} 2^{c(S)} ,\\
			\perm{G} &= \sum_{S \, \in \, \Sachs{G}} 2^{c(S)}.
		\end{align*}
		where $c(S)$ denotes the number of induced cycles of $S$, and $\kappa_{e}(S)$ denotes number of even connected components of \(S\).
	\end{theorem}
%
%
%
	
\section{Relevant structural  properties of matchable graphs}\label{Sec_Matching}

In this section, we assume that $G$ is a matchable graph, i.e \(G\) has at least a perfect matching. Let $v$ be a vertex of $G$ and let $M \in \mathcal{M}(G)$. We define the \emph{reachable set of $v$}, denoted by $\Rea{M}{v}$, as the set of vertices $u$ for which there exists an $M$-$mm$-alternating walk between $u$ and $v$ (see \cref{Figura1} and \cref{Figura5}).

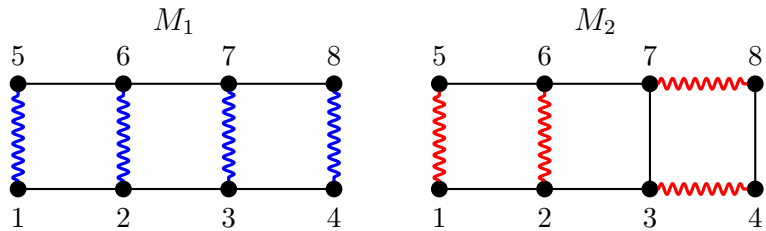
\begin{figure}[H]
	\centering
	\begin{tikzpicture}[commutative diagrams/every diagram,thick,scale=0.7] 
		
		
		\draw (2.5,3.5) node [anchor=north west][inner sep=0.75pt]    {$M_{1}$};
		
		\draw (10.5,3.5) node [anchor=north west][inner sep=0.75pt]    {$M_{2}$};
		
		\node[draw,black,fill,circle,scale=0.5,label=below:\small{$1$}] (1) at (0,0) { };
		\node[draw,black,fill,circle,scale=0.5,label=below:\small{$2$}] (2) at (2,0) { };
		\node[draw,black,fill,circle,scale=0.5,label=below:\small{$3$}] (3) at (4,0) { };
		\node[draw,black,fill,circle,scale=0.5,label=below:\small{$4$}] (4) at (6,0) { };
		\node[draw,black,fill,circle,scale=0.5,label=above:\small{$5$}] (5) at (0,2) { };
		\node[draw,black,fill,circle,scale=0.5,label=above:\small{$6$}] (6) at (2,2) { };
		\node[draw,black,fill,circle,scale=0.5,label=above:\small{$7$}] (7) at (4,2) { };
		\node[draw,black,fill,circle,scale=0.5,label=above:\small{$8$}] (8) at (6,2) { };
		
		\node[draw,black,fill,circle,scale=0.5,label=below:\small{$1$}] (11) at (8,0) { };
		\node[draw,black,fill,circle,scale=0.5,label=below:\small{$2$}] (21) at (10,0) { };
		\node[draw,black,fill,circle,scale=0.5,label=below:\small{$3$}] (31) at (12,0) { };
		\node[draw,black,fill,circle,scale=0.5,label=below:\small{$4$}] (41) at (14,0) { };
		\node[draw,black,fill,circle,scale=0.5,label=above:\small{$5$}] (51) at (8,2) { };
		\node[draw,black,fill,circle,scale=0.5,label=above:\small{$6$}] (61) at (10,2) { };
		\node[draw,black,fill,circle,scale=0.5,label=above:\small{$7$}] (71) at (12,2) { };
		\node[draw,black,fill,circle,scale=0.5,label=above:\small{$8$}] (81) at (14,2) { };
		
		\foreach \from/\to in {5/6,2/1,6/7,2/3,8/7,3/4,51/61,11/21,61/71,21/31,71/31,81/41} {
			\path[draw] (\from) -- (\to);}
		
		\foreach \from/\to in {1/5,2/6,7/3,8/4} {
			\path [draw, decorate, decoration={snake,amplitude=4,segment length=1.5mm, amplitude=2}, very thick, blue] (\from) -- (\to);
		}
		
		\foreach \from/\to in {11/51,21/61,71/81,31/41} {
			\path [draw, decorate, decoration={snake,amplitude=4,segment length=1.5mm, amplitude=2}, very thick, red] (\from) -- (\to);
		}

	\end{tikzpicture}
	\caption{The blue matching is $M_{1}$ and the red one is $M_{2}$. Note that $\Rea{M_1}{1} = \Rea{M_2}{1} = \{2,4,5,7\}$ and $\Rea{M_1}{2} = \Rea{M_2}{2} = \{1,3,6,8\}$}\label{Figura1}
\end{figure}

\begin{figure}[H]
	\centering
	\begin{tikzpicture}[commutative diagrams/every diagram,thick,scale=0.7] 
		
		
		\draw (3,4) node [anchor=north west][inner sep=0.75pt]    {$M_{1}$};
		
		\draw (11,4) node [anchor=north west][inner sep=0.75pt]    {$M_{2}$};
		
		\node[draw,black,fill,circle,scale=0.5,label=below:\small{$1$}] (1) at (0,0) { };
		\node[draw,black,fill,circle,scale=0.5,label=below:\small{$2$}] (2) at (2,0) { };
		\node[draw,black,fill,circle,scale=0.5,label=below:\small{$3$}] (3) at (4,0) { };
		\node[draw,black,fill,circle,scale=0.5,label=left:\small{$4$}] (4) at (0,2) { };
		\node[draw,black,fill,circle,scale=0.5,label=above right:\small{$5$}] (5) at (2,2) { };
		\node[draw,black,fill,circle,scale=0.5,label=above:\small{$6$}] (6) at (4,2) { };
		\node[draw,black,fill,circle,scale=0.5,label=above:\small{$7$}] (7) at (0,4) { };
		\node[draw,black,fill,circle,scale=0.5,label=above:\small{$8$}] (8) at (2,4) { };
		
		\node[draw,black,fill,circle,scale=0.5,label=below:\small{$1$}] (11) at (6+2,0) { };
		\node[draw,black,fill,circle,scale=0.5,label=below:\small{$2$}] (21) at (8+2,0) { };
		\node[draw,black,fill,circle,scale=0.5,label=below:\small{$3$}] (31) at (10+2,0) { };
		\node[draw,black,fill,circle,scale=0.5,label=left:\small{$4$}] (41) at (6+2,2) { };
		\node[draw,black,fill,circle,scale=0.5,label=above right:\small{$5$}] (51) at (8+2,2) { };
		\node[draw,black,fill,circle,scale=0.5,label=above:\small{$6$}] (61) at (10+2,2) { };
		\node[draw,black,fill,circle,scale=0.5,label=above:\small{$7$}] (71) at (6+2,4) { };
		\node[draw,black,fill,circle,scale=0.5,label=above:\small{$8$}] (81) at (8+2,4) { };
		
		\foreach \from/\to in {1/4,2/5,2/3,4/8,5/7,5/6,5/8,4/7,11/41,21/31,21/51,41/71,71/81,81/51,51/61,41/51} {
			\path[draw] (\from) -- (\to);}
		
		\foreach \from/\to in {1/2,4/5,6/3,8/7} {
			\path [draw, decorate, decoration={snake,amplitude=4,segment length=1.5mm, amplitude=2}, very thick, red] (\from) -- (\to);
		}
		
		\foreach \from/\to in {11/21,31/61,71/51,81/41} {
			\path [draw, decorate, decoration={snake,amplitude=4,segment length=1.5mm, amplitude=2}, very thick, blue] (\from) -- (\to);
		}

	\end{tikzpicture}
	\caption{The red matching is $M_{1}$ and the blue one is $M_{2}$. Note that $\Rea{M_1}{1} = \Rea{M_2}{1} = \{1,2,3,4,5,6,7,8\}$ and $\Rea{M_1}{2} = \Rea{M_2}{2} = \{1,2,3,4,5,6,7,8\}$}
	\label{Figura5}
\end{figure}
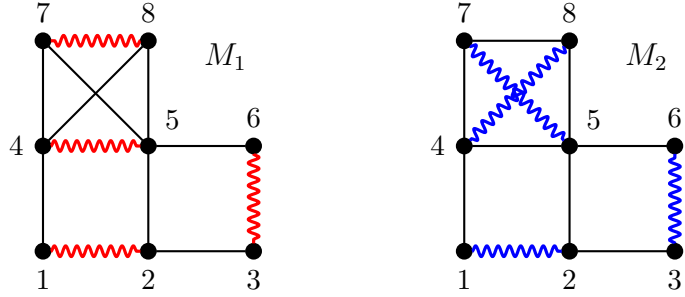

It is not a coincidence that in the examples of \cref{Figura1} and \cref{Figura5} we have $\Rea{M_1}{1} = \Rea{M_2}{1}$ and $\Rea{M_1}{2} = \Rea{M_2}{2}$. As we will see shortly, the set $\Rea{M}{v}$ does not depend on the choice of the perfect matching $M$. In contrast, for graphs without a perfect matching, the reachable set may depend on $M$; for instance, consider an odd cycle.

It is well known that the union of two perfect matchings, say $M_1$ and $M_2$, induces a graph whose connected components are either isolated edges or even cycles alternating between $M_1$ and $M_2$. From this fact, we deduce the following result.

\begin{lemma}[Reachability Lemma]\label[lemma]{asd1}
	Let $G$ be a matchable graph, and let $M_{1},M_{2} \in \mathcal{M}(G)$. Then $\Rea{M_1}{v} = \Rea{M_2}{v}$.
\end{lemma}

We need the following result from \cite{kevin2025posy}:
\begin{lemma}[\cite{kevin2025posy}]
	Let $G$ be a graph and let $M$ be a maximum matching of $G$. If $H$ is an $M$-Jposy of $G$, then for any $u,v \in V(H)$, there exists both an $M$-mm-alternating walk and an $M$-$nn$-alternating walk from $u$ to $v$.
\end{lemma}

\begin{corollary}\label[corollary]{v_Posy_vv}
	Let $G$ be a graph and let $M$ be a maximum matching of $G$. Let $v \in V(\fp{G})$ be a vertex that belongs to an $M$-Jposy. Then, there exists a $M$-$mm$-alternating closed walk starting at $v$.
\end{corollary}

\begin{theorem}\label{vwalk}
	Let $G$ be a matchable graph and let $M\in \mathcal{M}(G)$. Then, for each $v \in V(\fp{G})$, there exists an $M$-$mm$-alternating closed walk at $v$.
\end{theorem}
\begin{proof}
	Since $v\in V(\fp{G})$, there exists $M'\in \mathcal{M}(G)$ such that $v$ belongs to the vertices of a $M'$-Jposy. By \cref{v_Posy_vv}, we have that $v \in \Rea{M'}{v}$. By \cref{asd1}, we have that $\Rea{M'}{v} = \Rea{M}{v}$, i.e., there exists an $M$-$mm$-alternating closed walk starting at $v$.
\end{proof}

The following corollary shows that all SD-vertices of a matchable graph are ``marked'' by any perfect matching of $G$. 

Let $G$ be a matchable graph, $M \in \mathcal{M}(G)$, and $v \in V(\fp{G})$. By \cref{vwalk}, there exists an $M$–$mm$ alternating closed walk starting at $v$, and likewise an $M$–$mm$ alternating closed walk starting at $M(v)$.

\begin{corollary}\label[corollary]{asd12}
	Let $G$ be a matchable graph and let $M\in \mathcal{M}(G)$. Then, for each $v \in V(\fp{G})$ we have that $v$ is a vertex of an $M$-Jposy.
\end{corollary} 

For the KE-part of the graph in \cref{Figura1}, we have $\Rea{M_1}{1} \ne \{1,2,\dots,8\}$. In contrast, the graph in \cref{Figura5} has an empty KE-part, and in this case $\Rea{M_1}{1} = \{1,2,\dots,8\}$. The following theorem shows that this is always the case.

\begin{theorem}
	Let $G$ be a connected matchable graph and let $M \in \mathcal{M}(G)$. If $\ke{G} = \emptyset$, then for every $v \in V(G)$ we have
	\[
	\Rea{M}{v} = V(G).
	\]
\end{theorem} 

\begin{proof}
	If there exists $x \in V(G) \setminus R(M,v)$ such that $x \in N(R(M,v))$,
	then, by \cref{asd12}, there exists $y \in R(M,v)$ with $xy \in E(G) \setminus M$.
	Consider an $M$–$mm$-alternating walk from $v$ to $y$, and then extend it through the edge $xy$. By \cref{vwalk}, this yields an $M$–$mm$ alternating closed walk at $x$.
	Hence $x \in R(M,v)$, a contradiction.
\end{proof}

We need the following results from \cite{jaumeSDKE2025}:

\begin{theorem}[\cite{jaumeSDKE2025}]\label{teo_nomathcingedge}
	Let $G$ be a graph, $u\in V(\fp{G})$ and $v\in V(\ke{G})$. If $uv\in E(G)$, then $uv\notin M$ for every maximum matching $M$ of $G$.
\end{theorem}

\begin{corollary}[\cite{jaumeSDKE2025}]\label{mu}
	Let $G$ be a graph. Then $\mu(G) = \mu(\fp{G}) + \mu(\ke{G})$.
\end{corollary}

Thus, the previous results gives us an algorithm that, given a matchable graph $G$, can separate $\ke{G}$ from  $\fp{G}$.

\title{KE-P Algorithm for Graphs with Perfect Matching}

\begin{algorithm}[H]
	\DontPrintSemicolon
	\KwIn{Graph $G = (V,E)$ and a perfect matching $M$ of \(G\)}
	\KwOut{Sets $KE$ and $SD$, the SD-KE partition of $G$}
	\;
	Initialization: \;
	
	$KE \gets \emptyset$ \tcp*{Initial KE set}
	\;
	\While{$V\neq \emptyset$}{
		Chose \(v\in V\)\;
		\If{$\exists$ $M$-$mm$-alternating closed walk starting at $v$}{
			$V \gets V -v $ \tcp*{Remove the checked vertices}
		}\;
		\Else{
			$KE \gets KE \cup \{v, M(v)\}$ \tcp*{Add vertex and its partner}\;
			$V \gets V -\{v, M(v)\} $
		}
	}
	$SD \gets V(G) \setminus KE$ \tcp*{Define SD set}
	\Return $(SD, KE)$ \tcp*{Return the SD-KE partition of \(G\)}
	\caption{SD-KE separation algorithm}
\end{algorithm}

%

\begin{lemma}\label[lemma]{Lem_G_e_KE}
	Let $G$ be a graph, and let $e\in E(\ke{G})$. Then:
	$$V_{SD}(G) \subset V_{SD}(G-e).$$
	Even more, if there exists \(M \in \mathcal{M}(G)\) such that $e\notin M$, then  
	\[
	V_{SD}(G) = V_{SD}(G-e).
	\]
\end{lemma}
\begin{proof}
	If there exists \(M \in \mathcal{M}(G)\) such that $e\notin M$, then  \(\mathcal{M}(G-e)\subset \mathcal{M}(G)\). Therefore a vertex in a configuration of \(G-e\) is also a vertex in a configuration of \(G\). Hence, $V_{SD}(G-e) \subset V_{SD}(G)$. If \(v\) is an \(M\)-configuration of \(G\), by \cref{teo_nomathcingedge}, the matching \(M\cap E(\fp{G})\) can be extended to a maximum matching of \(G\) that do no use \(e\). Therefore, \(v\in V_{SD}(G-e)\). Hence,  \(V_{SD}(G) = V_{SD}(G-e)\).

	If $e\in \bigcap_{M\in \mathcal{M}} M$, then \(\mu(G-e)=\mu(G)-1\) 
	and \(M-e\) is a maximum matching of \(G-e\). Any configuration in \(G\) associated to \(M\) is a  configuration in \(G-e\) associated to \(M-e\). Therefore, $V_{SD}(G) \subset V_{SD}(G-e)$. This inclusion can be strict as the \cref{fig_kevin_1} shows.	 
\end{proof}

\begin{corollary}\label[corollary]{coro_G_e_KE}
	Let $G$ be a graph, and let $e\in E(\ke{G})$. If there exists \(M \in \mathcal{M}(G)\) such that $e\notin M$, then 
	\begin{align*} 
	\ke{G-e}&=\ke{G},\\
	\fp{G-e}&=\fp{G}.
	\end{align*}
\end{corollary}

\begin{figure}[H]
	\begin{subfigure}[b]{0.5\textwidth}
	\centering
	
	\begin{tikzpicture}[commutative diagrams/every diagram,thick,scale=0.7] 
		
		
		\node[draw,black,fill,circle,scale=0.5,label=below:\small{$1$}] (1) at (-2,0) { };
		\node[draw,black,fill,circle,scale=0.5,label=left:\small{$2$}] (2) at (-2,2) { };
		\node[draw,black,fill,circle,scale=0.5,label=below:\small{$3$}] (3) at (0,1) { };
		\node[draw,black,fill,circle,scale=0.5,label=below:\small{$4$}] (4) at (2,3) { };
		\node[draw,black,fill,circle,scale=0.5,label=above:\small{$5$}] (5) at (4,3) { };
		\node[draw,black,fill,circle,scale=0.5,label=right:\small{$6$}] (6) at (3,2) { };
		\node[draw,black,fill,circle,scale=0.5,label=right:\small{$7$}] (7) at (3,0) { };
		\node[draw,black,fill,circle,scale=0.5,label=above:\small{$8$}] (8) at (2,-1) { };
		\node[draw,black,fill,circle,scale=0.5,label=above:\small{$9$}] (9) at (4,-1) { };
		
		\foreach \from/\to in {3/2,3/1,3/8,8/7,4/5,4/6} {
			\path[draw] (\from) -- (\to);}
		
		\foreach \from/\to in {1/2,3/4,6/7,8/9} {
			\path [draw, thick, decorate, decoration={snake,amplitude=4,segment length=1.5mm, amplitude=2}, very thick, red] (\from) -- (\to);
		}
	\end{tikzpicture}
	\caption{\(V_{KE}(G)=\{6,7\}.\)}
	\end{subfigure}	
	%
	%
	%
	\begin{subfigure}[b]{0.5\textwidth}
		\centering
		
		\begin{tikzpicture}[commutative diagrams/every diagram,thick,scale=0.7] 
			
			
			\node[draw,black,fill,circle,scale=0.5,label=below:\small{$1$}] (1) at (-2,0) { };
			\node[draw,black,fill,circle,scale=0.5,label=left:\small{$2$}] (2) at (-2,2) { };
			\node[draw,black,fill,circle,scale=0.5,label=below:\small{$3$}] (3) at (0,1) { };
			\node[draw,black,fill,circle,scale=0.5,label=below:\small{$4$}] (4) at (2,3) { };
			\node[draw,black,fill,circle,scale=0.5,label=above:\small{$5$}] (5) at (4,3) { };
			\node[draw,black,fill,circle,scale=0.5,label=right:\small{$6$}] (6) at (3,2) { };
			\node[draw,black,fill,circle,scale=0.5,label=right:\small{$7$}] (7) at (3,0) { };
			\node[draw,black,fill,circle,scale=0.5,label=above:\small{$8$}] (8) at (2,-1) { };
			\node[draw,black,fill,circle,scale=0.5,label=above:\small{$9$}] (9) at (4,-1) { };
			
			\foreach \from/\to in {3/2,3/1,3/8,8/7,4/5,4/6} {
				\path[draw] (\from) -- (\to);}
			
			\foreach \from/\to in {1/2,3/4,8/9} {
				\path [draw, decorate, decoration={snake,amplitude=4,segment length=1.5mm, amplitude=2}, very thick, red] (\from) -- (\to);
			}
		\end{tikzpicture}
		\caption{\(V_{KE}(G-\{6,7\})=\emptyset.\)}
	\end{subfigure}

	\caption{Stability under edge deletion.}\label{fig_kevin_1}
\end{figure}
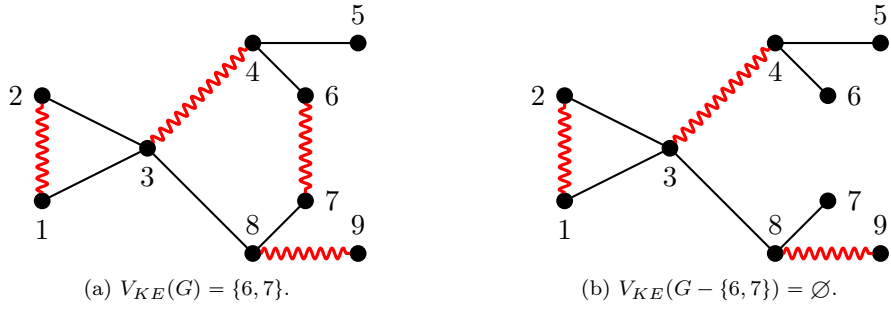
%
%
%
%
%
%

\section{Main result}\label{Main}

Let $G$ be a matchable graph. We define the \emph{SD-KE cut} of \(G\) as
\[
\partial_{SD\text{-}KE}(G):=\partial(\fp{G}) = E(G) - (E(\ke{G}) \cup E(\fp{G})).
\] 
The SD-KE cut of a graph is disjoint from any of its Sachs subgraphs.
\begin{theorem}\label{Main_result}
	If $G$ is an matchable graph and $S$ is a Sachs subgraph of $G$, then 
	\[
	\partial_{SD\text{-}KE}(G) \cap E(S) = \emptyset .
	\]
\end{theorem}
\begin{proof}
	Let \(G\) be an edge-minimal counterexample. Therefore there exist  a Sachs Subgraph \(S\) of \(G\) such that there exists an edge $uv \in \partial_{SD\text{-}KE}(G) \cap E(S)$ such that $u \in V(\ke{G})$ and $v \in V(\fp{G})$.
	
	Fix a perfect matching \(M\) of \(G\). By \cref{vwalk}, for every \(M\in \mathcal{M}(G)\) there exists \(W=vWv\), an \(M\)-\(mm\)-alternating closed walk starting in \(v\), we call it a semi-\(M\)-Jposy of \(v\).
	
	Let $uP$ be an \(MS\)-alternating path of maximal length starting at $u$.
	
	Let rewrite $uP$ as 
	\[
	uP = (u = u_{0})u_{1} \cdots (u_{k}=w).
	\]
	If $\deg(u_{k}) = 1$, then $u_{k-1} = M(u_{k})$, and thus $u_{k-1}u_{k}$ is an edge component of every Sachs subgraph of $G$ and belongs to every perfect matching of $G$. Moreover, since $uP$ is an $MS$-alternating path, the edge $u_{k-2}u_{k-1}$ is a Sachs edge. Consequently, $u_{k-2}, u_{k-1},$ and $u_{k}$ would all lie on a cycle component of any Sachs subgraph of $G$, which is absurd.
	
	Since \(uP\) is maximal, \(N_{S}=\{u_{j_{1}},\dots ,u_{j_{h}}\}\subset V(uP)\). If one of these indexes is odd, say \(j_{\ell}\) (note that \(k\) must be odd, so \(j_{\ell}\neq k-1\)), then \(uPwu_{j_{\ell}}P^{-1}u\) is an  \(M\)-\(mm\)-alternating closed walk starting in \(u\), i.e., a semi-\(M\)-Jposy of \(u\), see \cref{subcasei}. The semi-\(M\)-Jposy of \(w\), the edge \(uv\), and the semi-\(M\)-Jposy of \(u\), form an \(M\)-Jposy that contains \(u\), which is impossible.
	
	Assume that all the indices \(\{j_{1}, \dots,j_{h}\}\) are even. Then \(u_{j_{1}}Pwu_{j_{1}}\) is and even \(M\)-alternating cycle \(C\), see \cref{subcaseii}. Note that \(e_{1}=u_{j_{1}}u_{j_{1}+1}\) is a matching edge of \(uP\). Therefore \(M\triangle E(C)\) is a perfect matching of \(G\) that do not uses the edge  \(e_{1}\). By \cref{coro_G_e_KE}, \(\ke{G-e}=\ke{G}\). 
	
	 As the edges \(wu_{j_{1}}\) and \(u_{j_{1}}u_{j_{1}-1}\) are edges of the Sachs subgraph \(S\), the edge \(e_{1}\) is not an edge of \(S\). Hence, \(S\) is also a Sachs subgraph of \(G-e\). By minimality of \(G\),  \(uv\in \partial_{SD\text{-}KE}(G) \cap E(S) = \emptyset\) which is absurd.
\end{proof}

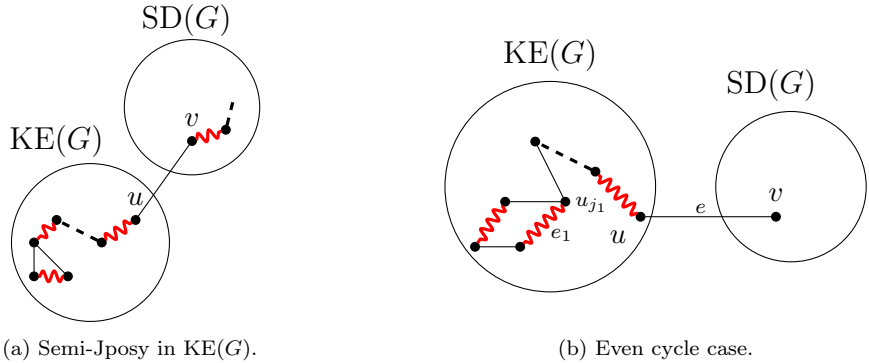
\begin{figure}[H]
	\centering
	\begin{subfigure}[b]{0.49\textwidth}
		\centering
		\begin{tikzpicture}[scale=0.3,commutative diagrams/every diagram]
			\node[draw,black,fill,circle,scale=0.3,label=above:\small{$u$}] (1) at (3.5,1)  {};
			\node[draw,black,fill,circle,scale=0.3] (2) at (2,0)  {};
			\node[] (3) at (0, -0.7)  {};
			\node[draw,black,fill,circle,scale=0.3,label=above:\small{$v$}] (5) at (6,4.5)  {};
			\node[draw,black,fill,circle,scale=0.3] (6) at (7.5,5)  {};
			\node[] (7) at (8 , 7)  {};	
			\node[draw,black,fill,circle,scale=0.3] (8) at (0,1)  {};
			\node[draw,black,fill,circle,scale=0.3] (9) at (-1,0)  {};
			\node[draw,black,fill,circle,scale=0.3] (10) at (-1,-1.5)  {};
			\node[draw,black,fill,circle,scale=0.3] (11) at (0.5,-1.5)  {};
			
			
			\foreach \from/\to in {2/1,5/6,8/9,11/10} {
				\path [draw, decorate, decoration={snake,amplitude=4,segment length=1.5mm, amplitude=2}, very thick, red] (\from) -- (\to);
			}
			
			\foreach \from/\to in {1/5,9/10,9/11}{
				\path[draw] (\from) -- (\to);
			}
			
			\foreach \from/\to in {6/7,2/8}{
				\draw[very thick, dashed] (\from) -- (\to);
			}
			
			\draw (1.5,0) circle (3.5);
			\node[above] at (0,3.3) {$\ke{G}$};
			\draw (6,6) circle (3);
			\node[above right] at (3.3,8.7) {$\fp{G}$};
		\end{tikzpicture}
		\caption{Semi-Jposy in \(\ke{G}\).}\label{subcasei}
	\end{subfigure}
	\hfill
	\begin{subfigure}[b]{0.49\textwidth}
		\centering
		\begin{tikzpicture}[scale=0.4,commutative diagrams/every diagram]
			\node[draw,black,fill,circle,scale=0.3,label=above:\small{$v$}] (0) at (10,5.5-2) { };
			\node[draw,black,fill,circle,scale=0.3,label=below left:\small{$u$}] (1) at (5.5,5.5-2) { };
			\node[draw,black,fill,circle,scale=0.3] (2) at (4,7-2) { };
			\node[draw,black,fill,circle,scale=0.3,label={[xshift=-2,yshift=-1pt]right:\scriptsize{$u_{j_{1}}$}}] (3) at (3,6.0-2) { };
			\node[draw,black,fill,circle,scale=0.3] (4) at (1.0,6.0-2) { };
			\node[draw,black,fill,circle,scale=0.3] (5) at (0.0,4.5-2) { };
			\node[draw,black,fill,circle,scale=0.3] (6) at (1.5,4.5-2) { };
			\node[draw,black,fill,circle,scale=0.3] (7) at (2,8-2) { };
			\node[] (8) at (12,5.5-2) { };
			\node[] (9) at (11,7.5-2) {};
			
			\foreach \from/\to in { {1/0}, {4/3}, {5/6}, {7/3}} {
				\path[draw] (\from) -- (\to);}
			
			\foreach \from/\to in { {4/5}, {3/6}, {1/2}} {
				\path [draw, decorate, decoration={snake,amplitude=4,segment length=1.5mm, amplitude=2}, very thick, red] (\from) -- (\to);
				}
			
			\foreach \from/\to in {2/7}{
				\draw[very thick, dashed] (\from) -- (\to);
			}
			
			\draw (2.5,6.5-2) circle (3.5);
			\node[above] at (2.5,10-2) {$\ke{G}$};
			
			\draw (10.5,6.5-2) circle (2.5);
			\node[above right] at (8,9-2) {$\fp{G}$};
			
			\node[label={[yshift=-1pt]above:\scriptsize{$e$}}] (12) at (7.5,5-2) { };
			\node[label={right:\scriptsize{$e_1$}}] (13) at (1.8,4.9-2) { };
		\end{tikzpicture}
		\caption{Even cycle case.}\label{subcaseii}
	\end{subfigure}
	\caption{Cases of proof of \cref{Main_result}}
	\label{combined}
\end{figure}

By \cref{Main_result} we have the following result.

\begin{corollary}
	Let G be a matchable graph, then
	\begin{align*}
	\det(G) & =\det(\fp{G})  \, \det(\ke{G}),\\
	\perm{G}& =\perm{\fp{G}} \, \perm{\ke{G}}.
	\end{align*}
\end{corollary}

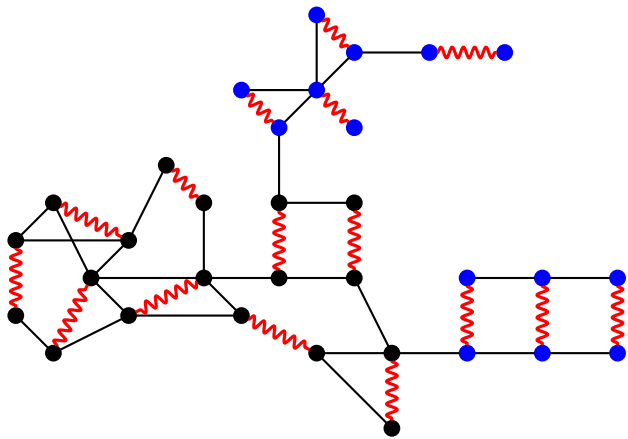
\begin{figure}[H]
	\centering
	\def \Vescala{1}
	\def \opacidad{0.2}
	\begin{tikzpicture}[commutative diagrams/every diagram,thick,scale=0.5] 
		
		\node[draw,black,fill,circle,scale=0.5] (0) at (-2,1) {};
		\node[draw,black,fill,circle,scale=0.5] (1) at (-4,2) {};
		\node[draw,black,fill,circle,scale=0.5] (2) at (-5,1) {};
		\node[draw,black,fill,circle,scale=0.5] (3) at (-5,-1) {};
		\node[draw,black,fill,circle,scale=0.5] (4) at (-4,-2) {};
		\node[draw,black,fill,circle,scale=0.5] (5) at (-3,0) {};
		\node[draw,black,fill,circle,scale=0.5] (6) at (-1,3) {};
		\node[draw,black,fill,circle,scale=0.5] (7) at (0,2) {};
		\node[draw,black,fill,circle,scale=0.5] (8) at (0,0) {};
		\node[draw,black,fill,circle,scale=0.5] (9) at (-2,-1) {};
		\node[draw,black,fill,circle,scale=0.5] (10) at (2,0) {};
		\node[draw,black,fill,circle,scale=0.5] (11) at (2,2) {};
		
		\node[draw,black,fill,circle,scale=0.5] (c) at (4,2) {};
		\node[draw,black,fill,circle,scale=0.5] (d) at (4,0) {};
		\node[draw,black,fill,circle,scale=0.5] (e) at (1,-1) {};
		\node[draw,black,fill,circle,scale=0.5] (f) at (3,-2) {};
		\node[draw,black,fill,circle,scale=0.5] (g) at (5,-2) {};
		\node[draw,black,fill,circle,scale=0.5] (d1) at (5,-4) {};
		\node[draw,blue,fill,circle,scale=0.5] (h) at (7,-2) {};
		\node[draw,blue,fill,circle,scale=0.5] (i) at (7,0) {};
		\node[draw,blue,fill,circle,scale=0.5] (j) at (9,0) {};
		\node[draw,blue,fill,circle,scale=0.5] (k) at (9,-2) {};
		\node[draw,blue,fill,circle,scale=0.5] (l) at (11,0) {};
		\node[draw,blue,fill,circle,scale=0.5] (m) at (11,-2) {};
		\node[draw,blue,fill,circle,scale=0.5] (r) at (2,4) {};
		\node[draw,blue,fill,circle,scale=0.5] (s) at (1,5) {};
		\node[draw,blue,fill,circle,scale=0.5] (t) at (3,5) {};
		\node[draw,blue,fill,circle,scale=0.5] (u) at (4,4) {};
		\node[draw,blue,fill,circle,scale=0.5] (v) at (3,7) {};
		\node[draw,blue,fill,circle,scale=0.5] (w) at (4,6) {};
		\node[draw,blue,fill,circle,scale=0.5] (z) at (6,6) {};
		\node[draw,blue,fill,circle,scale=0.5] (c1) at (8,6) {};
		
		\node[opacity=0] (phantom) at (0,-3) {};
		
		\foreach \from/\to in {1/2,1/5,3/4,0/6,7/8,9/5,9/4,0/5,5/8,2/0,10/8,8/e,9/e,11/c,10/d,d/g,f/g,f/d1,g/h,i/j,h/k,k/m,j/l,11/r,t/r,t/s,t/v,t/w,w/z} {
			\path[draw] (\from) -- (\to);
		}
		\foreach \from/\to in {2/3,4/5,1/0,6/7,8/9,10/11,e/f,c/d,g/d1,i/h,j/k,l/m,s/r,t/u,v/w,z/c1} {
			\path [draw, decorate, decoration={snake,amplitude=4,segment length=1.5mm, amplitude=2}, very thick, red] (\from) -- (\to);
		}
		
	\end{tikzpicture}
	\caption{A graph \(G\) with perfect matching. $\fp{G}$ and $\ke{G}$ are represented by black and blue vertices, respectively} \label{Exam_Det}
\end{figure}

For the graph $G$ in \cref{Exam_Det}, we have $\det(\ke{G}) = -1$ and $\det(\fp{G})= -5$. Therefore, $\det(G)=5$.

\section*{Acknowledgments}

This work was partially supported by Universidad Nacional de San Luis (Argentina), PROICO 03-0723, MATH AmSud, grant 22-MATH-02, Agencia I+D+i (Argentina), grants PICT-2020-Serie A-00549 and PICT-2021-CAT-II-00105, CONICET (Argentina) grant PIP 11220220100068CO.

\section*{Declaration of generative AI and AI-assisted technologies in the writing process}
During the preparation of this work the authors used ChatGPT-3.5 in order to improve the grammar of several paragraphs of the text. After using this service, the authors reviewed and edited the content as needed and take full responsibility for the content of the publication.

\section*{Data availability}

Data sharing not applicable to this article as no datasets were generated or analyzed during the current study.

\section*{Declarations}

\noindent\textbf{Conflict of interest} \ The authors declare that they have no conflict of interest.

\bibliographystyle{apalike}

\bibliography{TAGcitas_DPM}

\end{document}